\documentclass[9pt,oneside,leqno]{amsart}

\usepackage{amsmath,amsfonts,amssymb}
\usepackage[english]{babel}
\usepackage[utf8x]{inputenc}
\usepackage{url}

\usepackage[cm]{fullpage}

 \newtheorem{thm}{Theorem}[section]
 \newtheorem{prop}[thm]{Proposition}
 \newtheorem{cor}[thm]{Corollary}

 \theoremstyle{definition}
 \newtheorem{defi}[thm]{Definition}
 \newtheorem{rem}[thm]{Remark}

\newcommand{\N}{\mathbb{N}}
\newcommand{\Z}{\mathbb{Z}}

\newcommand{\C}{\mathbb{C}}

\newcommand{\del}{\partial}
\newcommand{\delbar}{\overline{\del}}

\newcommand{\paragrafon}[1]{\smallskip \paragraph{\texttt{Step {#1}}.\ }}

\DeclareMathOperator{\imm}{im}
\DeclareMathOperator{\de}{d}

\allowdisplaybreaks[1]

\title{Bott-Chern cohomology and $q$-complete domains}

\author{Daniele Angella}
\address[Daniele Angella]{Dipartimento di Matematica\\
Universit\`{a} di Pisa \\
Largo Bruno Pontecorvo 5, 56127\\ 
Pisa, Italy}
\email{angella@mail.dm.unipi.it}

\author{Simone Calamai}
\address[Simone Calamai]{Scuola Normale Superiore\\
Piazza dei Cavalieri 7, 56126\\
Pisa, Italy}
\email{simone.calamai@sns.it}

\thanks{This work was supported by the Project PRIN ``Varietà reali e complesse: geometria, topologia e analisi armonica'', by the Project FIRB ``Geometria Differenziale e Teoria Geometrica delle Funzioni'', and by GNSAGA of INdAM}

\begin{document}

\begin{abstract}
In studying the Bott-Chern and Aeppli cohomologies for $q$-complete manifolds, we introduce the class of cohomologically Bott-Chern $q$-complete manifolds.

\bigskip

Dans l'étude des cohomologies de Bott-Chern et d'Aeppli pour varietés $q$-completes, nous introduisons la classe des varietés cohomologiquement Bott-Chern $q$-completes.
\end{abstract}

\maketitle

\section*{Introduction}

The notion of \emph{$q$-complete} manifolds has been introduced and studied in \cite{andreotti-grauert, rothstein}. In particular, in \cite[Proposition 27]{andreotti-grauert}, A. Andreotti and H. Grauert proved a vanishing result for the higher-degree Dolbeault cohomology groups of $q$-complete manifolds $D$, namely, that $H^{r,s}_{\delbar}(D)=\{0\}$ for any $r\in\N$ and for any $s\geq q$, see also \cite[Theorem 5]{andreotti-vesentini}. Domains having such a vanishing property are called \emph{cohomologically $q$-complete}, and coincide with $q$-complete domains under some regularity conditions. In fact, M.~G. Eastwood and G. Vigna Suria proved that cohomologically $q$-complete domains of a Stein manifold with boundary of class $\mathcal{C}^2$ are in fact $q$-complete, \cite[Theorem 3.8]{eastwood-vignasuria}; see also \cite[\S V.5]{grauert-remmert} for a sheaf-theoretic characterization of Stein domains in $\C^n$.

Besides Dolbeault cohomology, other relevant tools to study geometry and analysis of complex manifolds $X$ are provided by Bott-Chern and Aeppli cohomologies, \cite{bott-chern, aeppli}, namely,
$$ H^{\bullet,\bullet}_{BC}(X) \;:=\; \frac{\ker\del\cap\ker\delbar}{\imm\del\delbar}\;, \qquad H^{\bullet,\bullet}_{A}(X) \;:=\; \frac{\ker\del\delbar}{\imm\del+\imm\delbar} \;. $$

\medskip

In this note, we are concerned with studying Bott-Chern and Aeppli cohomologies of open manifolds, in particular, assuming the vanishing of certain Dolbeault cohomology groups.

In particular, as a consequence of Theorem \ref{thm:bc-dr}, we get that cohomologically $q$-complete manifolds are also cohomologically Bott-Chern $q$-complete. Here, by \emph{cohomologically Bott-Chern $q$-complete} manifold, we mean a complex manifold $X$ of complex dimension $n$ such that $H^{r,s}_{BC}(X)$ vanishes for $r+s\geq n+q$. In this sense, such a notion provides a natural generalization of the classical notion of cohomologically $q$-complete manifolds.

\medskip

\noindent{\sl Acknowledgments.}
The authors are warmly grateful to Adriano Tomassini and Xiuxiong Chen for their constant support and encouragement.
They would like to thank also Alessandro Silva for useful discussions and for his interest, and Marco Franciosi for several suggestions.
Many thanks are also due to the anonymous Referee for his/her remarks.

\section{Dolbeault cohomology vanishing and Bott-Chern and Aeppli cohomologies}

In this section, we prove that the vanishing of certain Dolbeault cohomology groups assures connections between Bott-Chern cohomology and de Rham cohomology, and the vanishing of some Aeppli cohomology.

\medskip

Inspired by \cite{aeppli}, we prove the following inequalities, involving Bott-Chern and de Rham cohomologies, under the assumption of the vanishing of some Dolbeault cohomology groups. (As regards the compact case, we refer to \cite{angella-tomassini-3}, where an inequality {\itshape à la} Fr\"olicher is proven, yielding also a characterization of the $\del\delbar$-Lemma on compact complex manifolds.)

\begin{thm}\label{thm:bc-leq-geq-dr}
 Let $X$ be a complex manifold. Fix $(p,q)\in\left(\N\setminus\{0\}\right)^2$.
 \begin{enumerate}
  \item[{\itshape (a)}] If $\sum_{\substack{r+s=p+q-1\\s\geq \min\{p,q\}}} \dim_\C H^{r,s}_{\delbar}(X)=0$, then there is a natural injective map $H^{p,q}_{BC}(X) \to H^{p+q}_{dR}(X;\C)$.
  \item[{\itshape (b)}] If $\sum_{\substack{r+s=p+q\\s\geq\min\{p,q\}+1}} \dim_\C H^{r,s}_{\delbar}(X)=0$, then there is a natural surjective map $H^{p,q}_{BC}(X) \to H^{p+q}_{dR}(X;\C)$.
 \end{enumerate}
\end{thm}

\begin{proof}
We split the proof in the following steps.

\paragrafon{1} Consider the exact sequence
$$ 0 \to \mathcal{Z}^{p-1,q-1}_{\del\delbar} \to \mathcal{A}^{p-1,q-1} \stackrel{\del\delbar}{\to} \mathcal{Z}^{p,q}_{\de} \to 0 $$
of sheaves, \cite[Lemme 4.1(2.i)]{schweitzer}. Since the sheaf $\mathcal{A}^{p-1,q-1}$ is (a fine sheaf over a para-compact Hausdorff space and hence) acyclic, one gets the exact sequence
$$ \check H^0(X;\mathcal{A}^{p-1,q-1}) \stackrel{\del\delbar}{\to} \check H^0(X;\mathcal{Z}^{p,q}_{\de}) \to \check H^1(X;\mathcal{Z}^{p-1,q-1}_{\del\delbar}) \to 0 \;. $$
It follows that
\begin{equation}\label{eq:bc-leq-dr-1}
\check H^1(X;\mathcal{Z}^{p-1,q-1}_{\del\delbar}) \;\simeq\; \frac{\check H^0(X;\mathcal{Z}^{p,q}_{\de})}{\del\delbar \check H^0(X;\mathcal{A}^{p-1,q-1})} \;=\; H^{p,q}_{BC}(X) \;.
\end{equation}

\medskip

\paragrafon{2}
Consider the exact sequence
$$ 0 \to \mathcal{Z}^{p-1,q-1}_{\del} \to \mathcal{Z}^{p-1,q-1}_{\del\delbar} \stackrel{\del}{\to} \mathcal{Z}^{p,q-1}_{\de} \to 0 $$
of sheaves, because of the Dolbeault and Grothendieck Lemma, see, e.g., \cite[Lemma I.3.29]{demailly-agbook}.

\paragrafon{2a}
In case {\itshape (a)}, since $\check H^1(X;\mathcal{Z}^{p-1,q-1}_{\del}) \simeq H^{p,q-1}_{\del}(X) = \overline{H^{q-1,p}_{\delbar}(X)} = \{0\}$ by the hypothesis, one gets the injective map
\begin{equation}\label{eq:bc-leq-dr-2}
0 \to \check H^1(X;\mathcal{Z}^{p-1,q-1}_{\del\delbar}) \to \check H^1(X;\mathcal{Z}^{p,q-1}_{\de}) \;.
\end{equation}

\paragrafon{2b}
In case {\itshape (b)}, since $\check H^2(X;\mathcal{Z}^{p-1,q-1}_{\del}) \simeq H^{p+1,q-1}_{\del}(X) = \overline{H^{q-1,p+1}_{\delbar}(X)} = \{0\}$ by the hypothesis, one gets the surjective map
\begin{equation}\label{eq:bc-geq-dr-2}
\check H^1(X;\mathcal{Z}^{p-1,q-1}_{\del\delbar}) \to \check H^1(X;\mathcal{Z}^{p,q-1}_{\de}) \to 0 \;.
\end{equation}

\medskip

\paragrafon{3} Fix $\ell\in\left\{0,\ldots,q-2\right\}$. Consider the exact sequence
$$ 0 \to \mathcal{Z}^{p,\ell}_{\de} \to \mathcal{Z}^{p,\ell}_{\del} \stackrel{\delbar}{\to} \mathcal{Z}^{p,\ell+1}_{\de} \to 0 $$
of sheaves, \cite[Lemme 4.1(2.i, 2.$\overline{\text{ii}}$)]{schweitzer}.

\paragrafon{3a}
In case {\itshape (a)}, since $\check H^{q-\ell-1}(X;\mathcal{Z}^{p,\ell}_{\del}) \simeq H^{p+q-\ell-1, \ell}_{\del}(X) = \overline{H^{\ell, p+q-\ell-1}_{\delbar}(X)} = \{0\}$ by the hypothesis, one gets the injective map
$$ 0 \to \check H^{q-\ell-1}(X;\mathcal{Z}^{p,\ell+1}_{\de}) \to \check H^{q-\ell}(X;\mathcal{Z}^{p,\ell}_{\de}) \;. $$
Hence one gets the injective map
\begin{equation}\label{eq:bc-leq-dr-3}
0 \to H^1(X;\mathcal{Z}^{p,q-1}_{\de}) \to \check H^{q}(X;\mathcal{Z}^{p,0}_{\de}) \;.
\end{equation}

\paragrafon{3b}
In case {\itshape (b)}, since $\check H^{q-\ell}(X;\mathcal{Z}^{p,\ell}_{\del}) \simeq H^{p+q-\ell, \ell}_{\del}(X) = \overline{H^{\ell, p+q-\ell}_{\delbar}(X)} = \{0\}$ by the hypothesis, one gets the surjective map
$$ \check H^{q-\ell-1}(X;\mathcal{Z}^{p,\ell+1}_{\de}) \to \check H^{q-\ell}(X;\mathcal{Z}^{p,\ell}_{\de}) \to 0 \;. $$
Hence one gets the surjective map
\begin{equation}\label{eq:bc-geq-dr-3}
\check H^1(X;\mathcal{Z}^{p,q-1}_{\de}) \to \dim_\C \check H^{q}(X;\mathcal{Z}^{p,0}_{\de}) \to 0 \;.
\end{equation}

\medskip

\paragrafon{4} Fix $\ell\in\left\{0,\ldots,p-1\right\}$. Consider the exact sequence
$$ 0 \to \mathcal{Z}^{\ell,0}_{\de} \to \mathcal{Z}^{\ell,0}_{\delbar} \stackrel{\del}{\to} \mathcal{Z}^{\ell+1,0}_{\de} \to 0 $$
of sheaves, \cite[Lemme 4.1(2.ii)]{schweitzer}.

\paragrafon{4a}
In case {\itshape (a)}, since $\check H^{p+q-\ell-1}(X;\mathcal{Z}^{\ell,0}_{\delbar}) \simeq H^{\ell, p+q-\ell-1}_{\delbar}(X) = \{0\}$ by the hypothesis, one gets the injective map
$$ 0 \to \check H^{p+q-\ell-1}(X;\mathcal{Z}^{\ell+1,0}_{\de}) \to \check H^{p+q-\ell}(X;\mathcal{Z}^{\ell,0}_{\de}) \;. $$
Hence one gets the injective map
\begin{equation}\label{eq:bc-leq-dr-4}
0 \to \check H^q(X;\mathcal{Z}^{p,0}_{\de}) \to \check H^{p+q}(X;\underline{\C}_X) \;.
\end{equation}

\paragrafon{4b}
In case {\itshape (b)}, since $\check H^{p+q-\ell}(X;\mathcal{Z}^{\ell,0}_{\delbar}) \simeq H^{\ell, p+q-\ell}_{\delbar}(X) = \{0\}$ by the hypothesis, one gets the surjective map
$$ \check H^{p+q-\ell-1}(X;\mathcal{Z}^{\ell+1,0}_{\de}) \to \check H^{p+q-\ell}(X;\mathcal{Z}^{\ell,0}_{\de}) \to 0 \;. $$
Hence one gets the surjective map
\begin{equation}\label{eq:bc-geq-dr-4}
\dim_\C \check H^q(X;\mathcal{Z}^{p,0}_{\de}) \to \dim_\C \check H^{p+q}(X;\underline{\C}_X) \to 0 \;.
\end{equation}

\medskip

\paragrafon{5a}
In case {\itshape (a)}, by using \eqref{eq:bc-leq-dr-1}, \eqref{eq:bc-leq-dr-2}, \eqref{eq:bc-leq-dr-3}, and \eqref{eq:bc-leq-dr-4}, one gets
$$
H^{p,q}_{BC}(X)
\stackrel{\simeq}{\to}
\check H^1(X;\mathcal{Z}^{p-1,q-1}_{\del\delbar})
\hookrightarrow
\check H^1(X;\mathcal{Z}^{p,q-1}_{\de})
\hookrightarrow
\check H^{q}(X;\mathcal{Z}^{p,0}_{\de})
\hookrightarrow
\check H^{p+q}(X;\underline{\C}_X)
\stackrel{\simeq}{\to}
H^{p+q}_{dR}(X;\C) \;,
$$
concluding the proof of the item {\itshape (a)}.

\paragrafon{5b}
In case {\itshape (b)}, by using \eqref{eq:bc-leq-dr-1}, \eqref{eq:bc-geq-dr-2}, \eqref{eq:bc-geq-dr-3}, and \eqref{eq:bc-geq-dr-4}, one gets
$$
H^{p,q}_{BC}(X)
\stackrel{\simeq}{\to}
\check H^1(X;\mathcal{Z}^{p-1,q-1}_{\del\delbar})
\twoheadrightarrow
\check H^1(X;\mathcal{Z}^{p,q-1}_{\de})
\twoheadrightarrow
\check H^{q}(X;\mathcal{Z}^{p,0}_{\de})
\twoheadrightarrow
\check H^{p+q}(X;\underline{\C}_X)
\stackrel{\simeq}{\to}
H^{p+q}_{dR}(X;\C) \;,
$$
concluding the proof of the item {\itshape (b)}.
\end{proof}

As regards the Aeppli cohomology, we have the following vanishing result.

\begin{thm}\label{thm:a-dr}
 Let $X$ be a complex manifold. Fix $(p,q)\in\left(\N\setminus\{0\}\right)^2$. If,
 $$ H^{p,q}_{\delbar}(X) \;=\; \left\{0\right\} \qquad \text{ and } \qquad H^{q,p}_{\delbar}(X) \;=\; \left\{0\right\} \;, $$
 then
 $$ H^{p,q}_{A}(X) \;=\; \left\{0\right\} \;. $$
\end{thm}

\begin{proof}
We split the proof in the following steps.

\paragrafon{1} Consider the exact sequence
$$ 0 \to \mathcal{Z}^{p-1,q}_{\del} \oplus \mathcal{Z}^{p,q-1}_{\delbar} \to \mathcal{A}^{p-1,q} \oplus \mathcal{A}^{p,q-1} \stackrel{(\del,\delbar)}{\to} \mathcal{Z}^{p,q}_{\del\delbar} \to 0 $$
of sheaves, \cite[Lemme 4.1(3.i)]{schweitzer}. Since the sheaf $\mathcal{A}^{p-1,q} \oplus \mathcal{A}^{p,q-1}$ is (a fine sheaf over a para-compact Hausdorff space and hence) acyclic, one gets the exact sequence
$$ \check H^0(X;\mathcal{A}^{p-1,q} \oplus \mathcal{A}^{p,q-1}) \stackrel{(\del,\delbar)}{\to} \check H^0(X;\mathcal{Z}^{p,q}_{\del\delbar}) \to \check H^1(X;\mathcal{Z}^{p-1,q}_{\del} \oplus \mathcal{Z}^{p,q-1}_{\delbar}) \to 0 \;, $$
It follows that
\begin{equation}\label{eq:vanishing-A-1}
\check H^1(X;\mathcal{Z}^{p-1,q}_{\del} \oplus \mathcal{Z}^{p,q-1}_{\delbar}) \;\simeq\; \frac{\check H^0(X;\mathcal{Z}^{p,q}_{\del\delbar})}{\del \check H^0(X;\mathcal{A}^{p-1,q}) + \delbar \check H^0(X;\mathcal{A}^{p,q-1})} \;=\; H^{p,q}_{A}(X) \;.
\end{equation}

\paragrafon{2} Since $\check H^1(X;\mathcal{Z}^{p-1,q}_{\del}) \simeq H^{p,q}_{\del}(X) = \overline{H^{q,p}_{\delbar}(X)} = \{0\}$ and $\check H^1(X;\mathcal{Z}^{p,q-1}_{\delbar}) \simeq H^{p,q}_{\delbar}(X) = \{0\}$ by the hypotheses, one gets
\begin{equation}\label{eq:vanishing-A-2}
\check H^1(X;\mathcal{Z}^{p-1,q}_{\del} \oplus \mathcal{Z}^{p,q-1}_{\delbar}) \;=\; \{0\} \;.
\end{equation}

\paragrafon{3} By \eqref{eq:vanishing-A-1} and \eqref{eq:vanishing-A-2}, one gets the vanishing of $H^{p,q}_{A}(X)$.
\end{proof}

As a straightforward consequence, we get the following vanishing result for cohomologically $q$-complete manifold.

\begin{cor}
 Let $X$ be a cohomologically $q$-complete manifold. Then $H^{r,s}_{A}(X)=\{0\}$ for any $(r,s)\in\Z^2$ such that $\min\{r,s\}\geq q$.
\end{cor}

\section{Dolbeault cohomology vanishing and Bott-Chern and Aeppli cohomologies}

As partial converse of Theorem \ref{thm:bc-leq-geq-dr} and of Theorem \ref{thm:a-dr} respectively, we provide the following results.

\begin{prop}\label{prop:bc-vanish-dolb}
 Let $X$ be a complex manifold of complex dimension $n$. Fix $(p,q)\in \N\times\left(\N\setminus\{0\}\right)$. If
 $$ H^{p,q}_{BC}(X) \;=\; \{0\} \qquad \text{ and } \qquad H^{p+1,q}_{BC}(X) \;=\; \{0\} \;, $$
 then
 $$ H^{p,q}_{\delbar}(X) \;=\; \{0\} \;. $$
\end{prop}

\begin{proof}
 Take $\mathfrak{a} = [\alpha] \in H^{p,q}_{\delbar}(X)$, and consider $[\del\alpha] \in H^{p+1,q}_{BC}(X) = \{0\}$. (As a matter of notation, we set $\wedge^{\ell,m}X:=\{0\}$ for $\ell\not\in\N$ or $m\not\in\N$.) Hence there exists $\beta \in \wedge^{p,q-1}X$ such that $\del\alpha = \del\delbar\beta$. Consider $\left[\alpha-\delbar\beta\right] \in H^{p,q}_{BC}(X) = \{0\}$. Hence there exists $\gamma \in \wedge^{p-1,q-1}X$ such that $\alpha-\delbar\beta = \del\delbar\gamma$. Therefore $\alpha = \delbar\left(\beta-\del\gamma\right)$, that is, $\mathfrak{a} = 0 \in H^{p,q}_{\delbar}(X)$.
\end{proof}

\begin{prop}\label{prop:a-vanish-dolb}
 Let $X$ be a complex manifold of complex dimension $n$. Fix $(p,q)\in \N\times\left(\N\setminus\{0\}\right)$. If
 $$ H^{p-1,q}_{A}(X) \;=\; \{0\} \qquad \text{ and } \qquad H^{p,q}_{A}(X) \;=\; \{0\} \;, $$
 then
 $$ H^{p,q}_{\delbar}(X) \;=\; \{0\} \;. $$
\end{prop}

\begin{proof}
 Take $\mathfrak{a} = [\alpha] \in H^{p,q}_{\delbar}(X)$, and consider $[\alpha] \in H^{p,q}_{A}(X) = \{0\}$. (As a matter of notation, we set $\wedge^{\ell,m}X:=\{0\}$ for $\ell\not\in\N$ or $m\not\in\N$.) Hence there exist $\beta \in \wedge^{p-1,q}X$ and $\gamma \in \wedge^{p,q-1}X$ such that $\alpha = \del\beta + \delbar\gamma$. Consider $\left[\beta\right] \in H^{p-1,q}_{A}(X) = \{0\}$. Hence there exist $\xi \in \wedge^{p-2,q}X$ and $\eta \in \wedge^{p-1,q-1}X$ such that $\beta = \del\xi+\delbar \eta$. Therefore $\alpha = \delbar\left(\gamma-\del\eta\right)$, that is, $\mathfrak{a} = 0 \in H^{p,q}_{\delbar}(X)$.
\end{proof}

\begin{rem}
Note that, by Proposition \ref{prop:a-vanish-dolb}, respectively Proposition \ref{prop:bc-vanish-dolb}, and by \cite[Theorem V.5.2]{grauert-remmert}, for domains $D$ of $\C^n$ to be Stein, it is sufficient that $H^{0,\ell}_{A}(D) = \{0\}$ for any $\ell \in \left\{1,\ldots, n-1\right\}$, respectively that $H^{0,\ell}_{BC}(D) = H^{1,\ell}_{BC}(D) = \{0\}$ for any $\ell \in \left\{1,\ldots, n-1\right\}$, 	but not the converse.
\end{rem}

\section{Cohomologically Bott-Chern $q$-complete manifolds}

We recall that, fixed $q\in\N\setminus\{0\}$, a complex manifold $X$ is called \emph{cohomologically $q$-complete} if $H^{r,s}_{\delbar}(X)=\{0\}$ for any $r\in\N$ and for any $s\geq q$. In view of A. Andreotti and H. Grauert vanishing theorem, \cite[Proposition 27]{andreotti-grauert}, $q$-complete manifolds \cite{andreotti-grauert, rothstein} are cohomologically $q$-complete. Conversely, cohomologically $q$-complete domains of a Stein manifold with boundary of class $\mathcal{C}^2$ are $q$-complete, \cite[Theorem 3.8]{eastwood-vignasuria}.

In this section, we study the Bott-Chern counterpart of $q$-completeness. More precisely, consider the following definition.

\begin{defi}
 Let $X$ be a complex manifold of complex dimension $n$, and fix an integer $q\in \left\{ 1,\, \ldots ,\, n \right\}$. The manifold $X$ is called \emph{cohomologically Bott-Chern $q$-complete} if there holds that, for any positive integers $r$ and $s$ such that $r+s \geq n+q$, then $H^{r,s}_{BC}(X)=\{0\}$.
\end{defi}

\medskip

In order to motivate the previous definition, we note the following straightforward corollary of Theorem \ref{thm:bc-leq-geq-dr}.

\begin{cor}\label{thm:bc-dr}
 Let $X$ be a complex manifold. Fix $(p,q)\in\left(\N\setminus\{0\}\right)^2$. If,
 $$ \sum_{\substack{r+s=p+q-1\\s\geq \min\{p,q\}}} \dim_\C H^{r,s}_{\delbar}(X) \;=\; \sum_{\substack{r+s=p+q\\s\geq\min\{p,q\}+1}} \dim_\C H^{r,s}_{\delbar}(X) \;=\; 0 $$
 then, for any $(h,k)\in\left(\N\setminus\{0\}\right)^2$ such that $h+k=p+q$ and $\min\{p,q\}\leq h,k \leq \max\{p,q\}$, there is a natural isomorphism
 $$ H^{h,k}_{BC}(X) \;\simeq\; H^{p+q}_{dR}(X;\C) \;. $$
\end{cor}

In particular, for cohomologically $q$-complete manifolds, we get the following result.

\begin{cor}
 Let $X$ be a cohomologically $q$-complete manifold. Then $H^{r,s}_{BC}=H^{r+s}_{dR}(X;\C)$ for any $(r,s)\in\left(\N\setminus\{0\}\right)^2$ such that $\min\{r,s\}\geq q$.
\end{cor}

By using the Fr\"olicher inequality, \cite[Theorem 2]{frolicher}, see also \cite[Theorem 2.15]{mccleary}, we get the following vanishing result.

\begin{cor}\label{cor:vanishing-bc}
 Let $X$ be a complex manifold. Fix $(p,q)\in\left(\N\setminus\{0\}\right)^2$. If,
 $$ \sum_{\substack{r+s=p+q-1\\s\geq q}} \dim_\C H^{r,s}_{\delbar}(X) \;=\; \sum_{r+s=p+q} \dim_\C H^{r,s}_{\delbar}(X) \;=\; 0 \;, $$
 then
 $$ H^{p,q}_{BC}(X) \;\simeq\; \left\{0\right\} \;. $$
\end{cor}

\medskip

As an application, Corollary \ref{cor:vanishing-bc} relates the just introduced notion of cohomologically Bott-Chern $q$-completeness to the more classical notion of $q$-completeness.

\begin{cor}
 Every cohomologically $q$-complete manifold is also cohomologically Bott-Chern $q$-complete.
\end{cor}

\medskip

We conclude this note indicating some directions for further investigations.

\begin{rem}
 It would be interesting to have an example of a non-Stein domain being cohomologically Bott-Chern $1$-complete. More precisely, one would have a complex manifold $X$ of complex dimension $2$ such that $\sum_{r+s\geq3}\dim_\C H^{r,s}_{BC}(X)=0$, in particular, with $H^{r,s}_{\delbar}(X)=\{0\}$ for $(r,s)\in\left\{(1,1),\,(2,1),\,(0,2),\,(1,2),\,(2,2)\right\}$ and $H^{0,1}_{\delbar}(X)\neq \{0\}$. As the anonymous Referee pointed out to us, such an example can not occur when $X$ is a domain in $\C^2$ or, more generally, in a complex manifold whose holomorphic cotangent bundle is holomorphically trivial.
\end{rem}

\begin{rem}
 In view of the very definition of $q$-complete domains, \cite{andreotti-grauert, rothstein}, it would be interesting to have a geometric characterization of cohomologically Bott-Chern $q$-complete domains, for example in terms of positivity properties of the Levi form.
\end{rem}

\end{document}